\documentclass[a4paper]{amsart}
\usepackage{amssymb}
\usepackage{verbatim}
\usepackage{epic,eepic}
\usepackage[dvips]{graphicx}
\theoremstyle{plain}
  \newtheorem{thm}{Theorem}[section]
  \newtheorem{lem}[thm]{Lemma}
  
  \newtheorem{prop}[thm]{Proposition}
  \newtheorem{conj}[thm]{Conjecture}

  \newtheorem*{obs*}{Observation}
\theoremstyle{definition}

\theoremstyle{remark}
  \newtheorem{rem}[thm]{Remark}
  
  \newtheorem*{ack}{Acknowledgments}
\newcommand{\Z}{\mathbb{Z}}
\newcommand{\C}{\mathbb{C}}

\newcommand{\Vol}{\operatorname{Vol}}

\newcommand{\Res}{\operatorname{Res}}
\newcommand{\floor}[1]{\lfloor#1\rfloor}
\renewcommand{\Re}{\operatorname{Re}}
\renewcommand{\Im}{\operatorname{Im}}
\numberwithin{equation}{section}
\allowdisplaybreaks
\begin{document}
\title
{Colored Jones polynomials with polynomial growth}
\author{Kazuhiro Hikami}
\address{
Department of Physics,
Graduate School of Science,
University of Tokyo,
Hongo 7-3-1, Bunkyo, Tokyo 113-0033, Japan.}
\urladdr{http://gogh.phys.s.u-tokyo.ac.jp/{\textasciitilde}hikami/}
\email{hikami@phys.s.u-tokyo.ac.jp}
\author{Hitoshi Murakami}
\address{
Department of Mathematics,
Tokyo Institute of Technology,
Oh-okayama, Meguro, Tokyo 152-8551, Japan
}
\email{starshea@tky3.3web.ne.jp}
\date{\today}
\begin{abstract}
The volume conjecture and its generalizations say that the colored Jones polynomial corresponding to the $N$-dimensional irreducible representation of $sl(2;\C)$ of a (hyperbolic) knot evaluated at $\exp(c/N)$ grows exponentially with respect to $N$ if one fixes a complex number $c$ near $2\pi\sqrt{-1}$.
On the other hand if the absolute value of $c$ is small enough, it converges to the inverse of the Alexander polynomial evaluated at $\exp{c}$.
In this paper we study cases where it grows polynomially.
\end{abstract}
\begin{comment}
The volume conjecture and its generalizations say that the colored Jones polynomial corresponding to the N-dimensional irreducible representation of sl(2;C) of a (hyperbolic) knot evaluated at exp(c/N) grows exponentially with respect to N if one fixes a complex number c near 2*Pi*I.
On the other hand if the absolute value of c is small enough, it converges to the inverse of the Alexander polynomial evaluated at exp(c).
In this paper we study cases where it grows polynomially.
\end{comment}
\thanks{The first author is supported by Grant-in-Aid for Young Scientists (B) ( 18740227), and Grant-in-Aid for Scientific Research (B) (19340009) and (C) (19540069).
The second author is  supported by Grant-in-Aid for Exploratory Research (18654009).}
\maketitle
%%%%%%%%%%%%%%%%%%%%%%%%%%%%%%%%%%%%%%%%%%%%%%%%%%%%%%%%%%%%%%%%%%%%
\section{Introduction}
Let $J_N(K;q)$ be the $N$-dimensional colored Jones polynomial of a knot $K$ normalized so that $J_N(U;q)=1$ for the unknot $U$ and that $J_2(K;q)=V(K;q)$ is the original Jones polynomial \cite{Jones:BULAM385,Kirillov/Reshetikhin:1989}.
It is conjectured \cite{Kashaev:MODPLA95,Kashaev:LETMP97,Murakami/Murakami:ACTAM12001} that the sequence $\{J_N\bigl(K;\exp(2\pi\sqrt{-1}/N)\bigr)\}_{N=2,3,\dots}$ grows exponentially as $\exp\left(N\Vol\left(S^3\setminus{K}\right)/(2\pi)\right)$, where $\Vol$ is the simplicial volume normalized so that it coincides with the hyperbolic volume if $S^3\setminus{K}$ possesses the unique complete hyperbolic structure.
It is proved by Y.~Yokota and the second author that for the figure-eight knot the sequence $\{J_N\bigl(K;\exp(c/N)\bigr)\}_{N=2,3,\dots}$ also grows exponentially and gives the volumes and the Chern--Simons invariants obtained from $K$ by Dehn surgeries, if $c$ is close to $2\pi\sqrt{-1}$ \cite{Murakami/Yokota:JREIA2007}.
\par
On the other hand the second author proved that for the figure-eight knot the sequence $\{J_N\bigl(K;\exp(c/N)\bigr)\}_{N=2,3,\dots}$ converges to the inverse of the Alexander polynomial evaluated at $\exp{c}$, if $|c|$ is small \cite{Murakami:JPJGT2007}.
S.~Garoufalidis and T.~Le generalized it to every knot \cite{Garoufalidis/Le:aMMR}.
\par
Then one may wonder when the sequence $\{J_N\bigl(K;\exp(c/N)\bigr)\}_{N=2,3,\dots}$ grows polynomially.
The first example was given by R.~Kashaev and O.~Tirkkonen \cite{Kashaev/Tirkkonen:ZAPNS2000}.
In fact they showed that for torus knots the sequence $\{J_N\bigl(K;\exp(2\pi\sqrt{-1}/N)\bigr)\}_{N=2,3,\dots}$ grows as $N^{3/2}$.
The first author studied its relations to modular forms in \cite{Hikami/Kirillov:PHYLB2003,Hikami:EXPMA2003,Hikami:RAMJO2006}.
See also \cite{Hikami:COMMP2004} about a similar result for torus links.
\par
In this paper we give other examples where the sequence $\{J_N\bigl(K;\exp(c/N)\bigr)\}_{N=2,3,\dots}$ grows polynomially.
In fact we will give the cases where $\exp{c}$ is a zero of the Alexander polynomial such that $|c|$ is the smallest among all the logarithms of the zeroes of the Alexander polynomial.
\begin{thm}\label{thm:figure8}
Let $E$ be the figure-eight knot.
Put $\xi:=\log\bigl((3+\sqrt{5})/2\bigr)$.
Then
\begin{equation*}
  J_N\left(E;\exp\left(\frac{\xi}{N}\right)\right)
  \underset{N\to\infty}{\sim}
  \frac{\Gamma\left(\frac{1}{3}\right)}{(3\xi)^{\frac{2}{3}}}N^{\frac{2}{3}}.
\end{equation*}
\end{thm}
\begin{thm}\label{thm:torus}
Let $a$ and $b$ be coprime integers greater than two, and $T(a,b)$ the torus knot of type $(a,b)$.
Then
\begin{equation*}
  J_N\left(T(a,b);\exp\left(\frac{2\pi\sqrt{-1}}{abN}\right)\right)
  \underset{N\to\infty}{\sim}
  e^{-\frac{\pi\sqrt{-1}}{4}}\,
  \frac{\sin\left(\frac{\pi}{a}\right)\sin\left(\frac{\pi}{b}\right)}
       {\sqrt{2}\sin\left(\frac{\pi}{ab}\right)}
  N^{\frac{1}{2}}.
\end{equation*}
\end{thm}
\begin{rem}
Let us denote the Alexander polynomial of a knot $K$ by $\Delta(K;t)$.
Put $\Lambda(K):=\{z\in\C\mid\Delta(K;\exp{z})=0\}$; that is, $\Lambda(K)$ is the set of all the logarithmic values of the zeroes of the Alexander polynomial of $K$.
Since
\begin{gather*}
  \Delta(E;t)
  =
  -t+3-t^{-1}
  \\
  \intertext{and}
  \Delta\bigl(T(a,b)\bigr)
  =
  \frac{\left(t^{ab/2}-t^{-ab/2}\right)\left(t^{1/2}-t^{-1/2}\right)}
       {\left(t^{a/2}-t^{-a/2}\right)\left(t^{b/2}-t^{-b/2}\right)},
\end{gather*}
we have
\begin{gather*}
  \Lambda(E)
  =
  \{\pm\xi+2n\pi\sqrt{-1}\mid n\in\Z\}
  \\
  \intertext{and}
  \Lambda\bigl(T(a,b)\bigr)
  =
  \left\{
  \frac{2k\pi\sqrt{-1}}{ab}\mid k\in\Z,a\nmid k, b\nmid k
  \right\}.
\end{gather*}
Then we observe the following:
\begin{itemize}
\item
Since the figure-eight knot is amphicheiral, $J_N(E;q)=J_N(E;q^{-1})$.
Therefore we have the same formula for $-\xi$ in Theorem~\ref{thm:figure8}.
Note that $\xi,-\xi\in\Lambda(E)$ and that $|\xi|=|-\xi|=\min\{|z|\mid z\in\Lambda(E)\}$.
\item
Since $J_N(K;\overline{q})=\overline{J_N(K;q)}$ for any knot $K$, we have
\begin{equation*}
  J_N\left(T(a,b);\exp\left(-\frac{2\pi\sqrt{-1}}{abN}\right)\right)
  \underset{N\to\infty}{\sim}
  e^{\frac{\pi\sqrt{-1}}{4}}\,
  \frac{\sin\left(\frac{\pi}{a}\right)\sin\left(\frac{\pi}{b}\right)}
       {\sqrt{2}\sin\left(\frac{\pi}{ab}\right)}
  N^{\frac{1}{2}},
\end{equation*}
where $\overline{z}$ is the complex conjugate of $z$.
Note that $\pm2\pi\sqrt{-1}/(ab)\in\Lambda(T(a,b))$ and that $|\pm2\pi\sqrt{-1}/(ab)|=\min\{|z|\mid z\in\Lambda(T(a,b))\}$.
\end{itemize}
\end{rem}
In \cite{Murakami:JPJGT2007} the second author proved that if $|c|$ is small enough, then
\begin{equation*}
  \lim_{N\to\infty}J_N\left(E;\exp\left(\frac{c}{N}\right)\right)
  =
  \frac{1}{\Delta(E;\exp{c})}.
\end{equation*}
Moreover S.~Garoufalidis and T.~Le \cite{Garoufalidis/Le:aMMR} proved the formula above holds for any knot.
\par
So it is natural to expect that the sequence
$\{J_N\bigl(K;\exp(c/N)\bigr)\}_{N=2,3,\dots}$ diverges for $c\in\Lambda(K)$.
\begin{conj}\label{conj:polynomial}
For any knot $K$, let $c$ be an element in $\Lambda(K)$ such that $|c|=\min\{|z|\mid z\in\Lambda(K)\}$.
Then the sequence $\left\{J_N\bigl(K;\exp(c/N)\bigr)\right\}_{N=2,3,\dots}$ grows polynomially, and
\begin{equation*}
  \lim_{N\to\infty}J_N\left(K;\exp\left(\frac{tc}{N}\right)\right)
  =
  \frac{1}{\Delta\bigl(K;\exp(tc)\bigr)}
\end{equation*}
for $0\le t<1$.
\end{conj}
We will discuss some evidence for the conjecture by using connected-sums.
\begin{ack}
The second author thanks S.~Garoufalidis, A.~Gibson and R.~van der Veen for helpful discussions.
He also thanks the Institute of Mathematics, Vietnamese Academy of Science and the organizers of the conference `International Conference on Quantum Topology', 6--12 August, 2007 for their hospitality.
\end{ack}
%%%%%%%%%%%%%%%%%%%%%%%%%%%%%%%%%%%%%%%%%%%%%%%%%%%%%%%%%%%%%%%%%
\section{Figure-eight knot}
In this section we prove Theorem~\ref{thm:figure8}.
\par
By K.~Habiro \cite{Habiro:SURIK00} (see also \cite{Masbaum:ALGGT12003}) and T.~Le, it is proved that
\begin{align*}
  J_N(E;q)
  &=
  \sum_{k=0}^{N-1}
  \prod_{j=1}^{k}
  \left(q^{(N+j)/2}-q^{-(N+j)/2}\right)
  \left(q^{(N-j)/2}-q^{-(N-j)/2}\right).
\end{align*}
Replacing $q$ with $\exp(\xi/N)$ we get
\begin{equation}\label{eq:fig8}
  J_N\left(E;\exp\frac{\xi}{N}\right)
  =
  \sum_{k=0}^{N-1}
  \prod_{j=1}^{k}
  f\left(\frac{j}{N}\right),
\end{equation}
where $f(x):=3-2\cosh(\xi x)$.
We will approximate $J_N\bigl(E;\exp(\xi/N)\bigr)$ as follows.
\begin{prop}\label{prop:approx_fig8}
For any $0<\varepsilon<1$, we have
\begin{equation*}
  J_N\left(E;\exp\frac{\xi}{N}\right)
  \underset{N\to\infty}{\sim}
  N\int_{0}^{\varepsilon}
  \exp
  \left(
    N\int_{0}^{y}\log{f(x)}dx
  \right)
  dy.
\end{equation*}
\end{prop}
To prove Proposition~\ref{prop:approx_fig8}, we first show that even if we restrict the summation range of the right hand side of \eqref{eq:fig8} to a smaller one, the difference is very small.
\begin{lem}\label{lem:epsilon}
For any $0<\varepsilon<1$, we have
\begin{equation*}
  \sum_{k=0}^{N-1}
  \prod_{j=1}^{k}
  f\left(\frac{j}{N}\right)
  \underset{N\to\infty}{\sim}
  \sum_{k=0}^{\floor{\varepsilon N}}
  \prod_{j=1}^{k}
  f\left(\frac{j}{N}\right),
\end{equation*}
where $\floor{x}$ denotes the largest integer that does not exceed $x$.
\end{lem}
\begin{proof}
Since $0<f(x)<1$ for $0<x<1$, if $k>\varepsilon N$ we have
\begin{equation*}
  \prod_{j=1}^{k}
  f\left(\frac{j}{N}\right)
  <
  \prod_{j=1}^{\floor{\varepsilon N}}
  f\left(\frac{j}{N}\right).
\end{equation*}
Therefore
\begin{equation*}
\begin{split}
  0&<
  \sum_{k=0}^{N-1}
  \prod_{j=1}^{k}
  f\left(\frac{j}{N}\right)
  -
  \sum_{k=0}^{\floor{\varepsilon N}}
  \prod_{j=1}^{k}
  f\left(\frac{j}{N}\right)
  =
  \sum_{k=\floor{\varepsilon N}+1}^{N-1}
  \prod_{j=1}^{k}
  f\left(\frac{j}{N}\right)
  \\
  &<
  \sum_{k=\floor{\varepsilon N}+1}^{N-1}
  \prod_{j=1}^{\floor{\varepsilon N}}f\left(\frac{j}{N}\right)
  =
  (N-\floor{\varepsilon N}-1)
  \prod_{j=1}^{\floor{\varepsilon N}}f\left(\frac{j}{N}\right).
  \\
  &<
  N\prod_{j=1}^{\floor{\varepsilon N}}f\left(\frac{j}{N}\right).
\end{split}
\end{equation*}
For any $\varepsilon'$ with $0<\varepsilon'<\varepsilon$ we have
\begin{equation*}
\begin{split}
  \prod_{j=1}^{\floor{\varepsilon N}}f\left(\frac{j}{N}\right)
  &=
  \prod_{j=1}^{\floor{\varepsilon'N}}f\left(\frac{j}{N}\right)
  \times
  \prod_{j=\floor{\varepsilon'N}+1}^{\floor{\varepsilon N}}f\left(\frac{j}{N}\right)
  \\
  &<
  \prod_{j=1}^{\floor{\varepsilon'N}}f\left(\frac{j}{N}\right)
  \times
  f\left(\frac{\floor{\varepsilon'N}+1}{N}\right)^{\floor{\varepsilon N}-\floor{\varepsilon'N}}
  \\
  &<
  f\left(\frac{\floor{\varepsilon'N}+1}{N}\right)^{\floor{\varepsilon N}-\floor{\varepsilon'N}}
\end{split}
\end{equation*}
since $f$ is a decreasing function and $0<f(x)<1$ for $0<x<1$.
Since $\floor{\varepsilon'N}+1>\varepsilon'N$,
\begin{equation*}
  f\left(\frac{\floor{\varepsilon'N}+1}{N}\right)
  <
  f(\varepsilon').
\end{equation*}
Therefore we have
\begin{equation*}
  \sum_{k=0}^{N-1}
  \prod_{j=1}^{k}
  f\left(\frac{j}{N}\right)
  -
  \sum_{k=0}^{\floor{\varepsilon N}}
  \prod_{j=1}^{k}
  f\left(\frac{j}{N}\right)
  <
  Nf(\varepsilon')^{\floor{\varepsilon N}-\floor{\varepsilon'N}}
  <
  Nf(\varepsilon')^{(\varepsilon-\varepsilon')N-1}
\end{equation*}
since $\floor{\varepsilon N}+1>\varepsilon N$, $\floor{\varepsilon'N}\le\varepsilon'N$, and $0<f(\varepsilon')<1$.
\par
Now since 
$\displaystyle
\sum_{k=0}^{\floor{\varepsilon N}}\prod_{j=1}^{k}f\left(\frac{j}{N}\right)\ge1$,
we have
\begin{equation*}
\begin{split}
  \left|
    \frac%
      {\displaystyle
        \sum_{k=0}^{N-1}
        \prod_{j=1}^{k}
        f\left(\frac{j}{N}\right)
      }
      {\displaystyle
        \sum_{k=0}^{\floor{\varepsilon N}}
        \prod_{j=1}^{k}
        f\left(\frac{j}{N}\right)
      }
    -1
  \right|
  &=
  \frac%
    {\displaystyle
      \sum_{k=0}^{N-1}
      \prod_{j=1}^{k}
      f\left(\frac{j}{N}\right)
      -\sum_{k=0}^{\floor{\varepsilon N}}
      \prod_{j=1}^{k}
      f\left(\frac{j}{N}\right)
    }
    {\displaystyle
      \sum_{k=0}^{\floor{\varepsilon N}}
      \prod_{j=1}^{k}
      f\left(\frac{j}{N}\right)
    }
  \\
  &<
  \frac%
    {\displaystyle
      Nf(\varepsilon')^{(\varepsilon-\varepsilon')N-1}
    }
    {\displaystyle
      \sum_{k=0}^{\floor{\varepsilon N}}
      \prod_{j=1}^{k}
      f\left(\frac{j}{N}\right)
    }
  \le
  Nf(\varepsilon')^{(\varepsilon-\varepsilon')N-1}.
\end{split}
\end{equation*}
Since $0<f(\varepsilon')<1$, the rightmost side can be arbitrarily small as $N$ grows, proving the desired asymptotic formula.
\end{proof}
Next we replace the summation with a Riemann integral.
\begin{lem}\label{lem:Riemann}
For any $0<\varepsilon<1$, we have
\begin{multline*}
  f(\varepsilon)
  \int_{0}^{\varepsilon}
  \exp
  \left(
    N\int_{0}^{y}\log{f(x)}dx
  \right)
  dy
  <
  \frac{1}{N}
  \sum_{k=0}^{\floor{\varepsilon N}}
  \prod_{j=1}^{k}f\left(\frac{j}{N}\right)
  \\
  <
  \int_{0}^{\varepsilon}
  \exp
  \left(
    N\int_{0}^{y}\log{f(x)}dx
  \right)
  dy
  +
  \frac{1}{N}.
\end{multline*}
\end{lem}
\begin{proof}
First of all we recall $f(x)=3-2\cosh(\xi x)$, and note that the function $\log{f(x)}$ is
\begin{itemize}
\item
  negative, since $0<f(x)<1$,
\item
  decreasing, since
  $d\,\log{f(x)}/d\,x=-2\xi\sinh(\xi x)/f(x)<0$,
\end{itemize}
for $0<x<1$.
(Note also that $f(0)=1$ and $f(1)=3-2\cosh(\xi)=0$.)
\par
Therefore we have
\begin{equation*}
\begin{split}
  0>
  \frac{1}{N}\sum_{j=1}^{k}\log{f\left(\frac{j}{N}\right)}
  -
  \int_{0}^{k/N}\log{f(x)}dx
  &>
  \frac{1}{N}\log{f\left(\frac{k}{N}\right)}
  \\
  &\ge
  \frac{1}{N}\log{f(\varepsilon)}
\end{split}
\end{equation*}
for $k\le\varepsilon N$.
\par
Multiplying by $N$ and taking exponential, we have
\begin{equation*}
  1
  >
  \frac%
    {\displaystyle\prod_{j=1}^{k}f\left(\frac{j}{N}\right)}
    {
      \exp\left(N\displaystyle\int_{0}^{k/N}\log{f(x)}dx\right)
    }
  >
  f(\varepsilon),
\end{equation*}
and so
\begin{multline}\label{ineq:log}
  f(\varepsilon)\frac{1}{N}
  \sum_{k=0}^{\floor{\varepsilon N}}
  \exp\left(N\displaystyle\int_{0}^{k/N}\log{f(x)}dx\right)
  <
  \frac{1}{N}
  \sum_{k=0}^{\floor{\varepsilon N}}
  \prod_{j=1}^{k}f\left(\frac{j}{N}\right)
  \\
  <
  \frac{1}{N}
  \sum_{k=0}^{\floor{\varepsilon N}}
  \exp\left(N\displaystyle\int_{0}^{k/N}\log{f(x)}dx\right).
\end{multline}
Since the function $\exp\left(N\int_{0}^{y}\log{f(x)}dx\right)$ of $y$
is positive and decreasing for $0<y<1$, we have
\begin{equation*}
\begin{split}
  0
  &<
  \frac{1}{N}
  \sum_{k=0}^{\floor{\varepsilon N}}
  \exp
  \left(
    N\int_{0}^{k/N}\log{f(x)}dx
  \right)
  -
  \int_{0}^{\varepsilon}
  \exp
  \left(
    N\int_{0}^{y}\log{f(x)}dx
  \right)
  dy
  \\
  &<
  \frac{1}{N}
  \left(
    1-
    \exp\left(N\int_{0}^{\varepsilon}\log{f(x)}dx\right)
  \right)
  <
  \frac{1}{N}.
\end{split}
\end{equation*}
From \eqref{ineq:log} we have
\begin{multline*}
  f(\varepsilon)
  \int_{0}^{\varepsilon}
  \exp
  \left(
    N\int_{0}^{y}\log{f(x)}dx
  \right)
  dy
  <
  \frac{1}{N}
  \sum_{k=0}^{\floor{\varepsilon N}}
  \prod_{j=1}^{k}f\left(\frac{j}{N}\right)
  \\
  <
  \int_{0}^{\varepsilon}
  \exp
  \left(
    N\int_{0}^{y}\log{f(x)}dx
  \right)
  dy
  +
  \frac{1}{N},
\end{multline*}
completing the proof.
\end{proof}
Now we want to calculate the asymptotic behavior of the integral appearing in Lemma~\ref{lem:Riemann}.
To do that, recall Laplace's method \cite[\S~2.4]{Erdelyi:1956} to study asymptotic behaviors.
\begin{prop}\label{prop:Laplace}
Let us consider the following integral:
\begin{equation*}
  \int_{\alpha}^{\beta}g(t)e^{Nh(t)}\,dt,
\end{equation*}
where $h$ is a (suitably) differentiable real function and $g$ is a continuous complex function.
If $h'(\alpha)=h''(\alpha)=0$, $h^{(3)}(\alpha)<0$ and $h'(t)<0$ for $\alpha<t\le\beta$, then we have
\begin{equation*}
  \int_{\alpha}^{\beta}g(t)e^{Nh(t)}\,dt
  \underset{N\to\infty}{\sim}
  g(\alpha)
  \Gamma\left(\frac{1}{3}\right)
  \left(\frac{-2}{9h^{(3)}(\alpha)N}\right)^{1/3}
  e^{Nh(\alpha)}.
\end{equation*}
\end{prop}
We use Proposition~\ref{prop:Laplace} for
$h(y):=\int_{0}^{y}\log{f(x)}dx$ with $f(x)=3-2\cosh(\xi x)$ and $g(x):=1$.
Since
\begin{itemize}
\item
$h'(y)=\log{f(y)}$, $h''(y)=f'(y)/f(y)=-2\xi\sinh(\xi y)/(f(y))$ and
\item
$h^{(3)}(0)=f''(0)f(0)-f'(0)^2/f(0)^2=-2\xi^2$,
\end{itemize}
we have
\begin{multline}\label{eq:Gamma}
  \int_{0}^{\varepsilon}
  \exp
  \left(
    N\int_{0}^{y}\log\bigl(f(x)\bigr)dx
  \right)
  dy
  \\
  \underset{N\to\infty}{\sim}
  \Gamma\left(\frac{1}{3}\right)
  \left(
    \frac{2}{18\xi^2N}
  \right)^{1/3}
  =
  \frac{\Gamma\left(\frac{1}{3}\right)}{(3\xi)^{2/3}}N^{-1/3}.
\end{multline}
\par
Now we can prove Proposition~\ref{prop:approx_fig8}.
\begin{proof}[Proof of Proposition~\ref{prop:approx_fig8}]
We have from Lemma~\ref{lem:Riemann}
\begin{multline*}
  f(\varepsilon)
  <
  \frac{\displaystyle
        \sum_{k=0}^{\floor{\varepsilon N}}
        \prod_{j=1}^{k}f\left(\frac{j}{N}\right)}
       {\displaystyle
        N\int_{0}^{\varepsilon}
        \exp
        \left(
          N\int_{0}^{y}\log{f(x)}dx
        \right)
        dy}
  <
  1+
  \frac{1}
       {\displaystyle
        N\int_{0}^{\varepsilon}
        \exp
        \left(
          N\int_{0}^{y}\log{f(x)}dx
        \right)
        dy}.
\end{multline*}
Since $f(\varepsilon)=1$ when $\varepsilon\to0$, and the denominators go to infinity when $N\to0$ from \eqref{eq:Gamma}, we have
\begin{equation*}
  \sum_{k=0}^{\floor{\varepsilon N}}
  \prod_{j=1}^{k}f\left(\frac{j}{N}\right)
  \underset{N\to\infty}{\sim}
  N
  \int_{0}^{\varepsilon}
  \exp
  \left(
    N\int_{0}^{y}\log{f(x)}dx
  \right)
  dy.
\end{equation*}
From Lemma~\ref{lem:epsilon} we have Proposition~\ref{prop:approx_fig8}.
\end{proof}
\par
Now we have Theorem~\ref{thm:figure8} from \eqref{eq:Gamma}.
%%%%%%%%%%%%%%%%%%%%%%%%%%%%%%%%%%%%%%%%%%%%%%%%%%%%%%%%%%%%%%%%%%%%%%%%%
\section{Torus knots}
In this section we prove Theorem~\ref{thm:torus}.
\par
Let $T(a,b)$ be the torus knot of type $(a,b)$ for integers $a$ and $b$ with $(a,b)=1$, $a>1$ and $b>1$.
Then we will show more general results as follows.
\begin{thm}\label{thm:Hikami}
Let $r$ be a real number.
\par
If $|r|<1/(ab)$ then
\begin{equation*}
  \lim_{N\to\infty}J_N\left(T(a,b);\exp(2\pi{r}\sqrt{-1}/N)\right)
  =
  \frac{1}{\Delta\left(T(a,b);\exp(2\pi{r}\sqrt{-1})\right)},
\end{equation*}
and if $r=1/(ab)$ then
\begin{equation*}
  J_N\left(T(a,b);\exp(2\pi{r}\sqrt{-1}/N)\right)
  \underset{N\to\infty}{\sim}
  e^{-\pi\sqrt{-1}/4}
  \frac{\sin(\pi/a)\sin(\pi/b)}{\sqrt{2}\sin(\pi/(ab))}
  \sqrt{N}.
\end{equation*}
\end{thm}
In fact we will prove yet more general results.
\par
Put
\begin{equation*}
  \tau_{K}(z)
  :=
  \frac{2\sinh{z}}{\Delta\left(K;e^{2z}\right)}
\end{equation*}
for a knot $K$.
Note that
\begin{equation*}
  \tau_{T(a,b)}(z)
  =
  \frac{2\sinh(az)\sinh(bz)}{\sinh(abz)}.
\end{equation*}
Let $\mathcal{P}$ be the set of poles of $\tau_{T(a,b)}(z)$, that is,
\begin{equation*}
  \mathcal{P}
  :=
  \left\{
    \frac{k\pi\sqrt{-1}}{ab}
    \Big|
    k\in\Z,
    a\nmid k,b\nmid k
  \right\}.
\end{equation*}
We define $\gamma_{a,b,r}(k)$ to be $(2k)!$ times the coefficient of the Laurent expansion of $\tau_{T(a,b)}(z)$ around $z=\pi{r}\sqrt{-1}$ of degree $2k$.
Note that it is equal to $d^{2k}\,\tau_{T(a,b)}(w)/d\,w^{2k}\bigr|_{w=\pi{r}\sqrt{-1}}$ if $\pi{r}\sqrt{-1}\not\in\mathcal{P}$.
Put
\begin{align*}
  \mathcal{R}\bigl(T(a,b);j\bigr)
  &:=
  \frac{4}{ab}\sin\left(\frac{j\pi}{a}\right)\sin\left(\frac{j\pi}{b}\right),
  \\
  \intertext{and}
  \mathcal{CS}\bigl(T(a,b);j\bigr)
  &:=
  \exp\left(\frac{j^2\pi\sqrt{-1}}{2ab}\right).
\end{align*}
Then we have the following asymptotic expansions.
\begin{prop}\label{prop:asymptotic}
Assume that $r$ is a real number but not an integer.
If $\pi{r}\sqrt{-1}\not\in\mathcal{P}$, then
\begin{equation}\label{eq:in_P}
\begin{split}
  &J_N\left(T(a,b);\exp\left(\frac{2\pi{r}\sqrt{-1}}{N}\right)\right)
  \\
  =&
  \frac{e^{\left(ab-\frac{a}{b}-\frac{b}{a}\right)\frac{\pi{r}\sqrt{-1}}{2N}}}
       {\sin(\pi{r})}
  \left\{
    \frac{1}{2\sqrt{-1}}
    \sum_{k=0}^{\infty}
    \frac{\gamma_{a,b,r}(k)}{k!}
    \left(\frac{\pi{r}\sqrt{-1}}{2abN}\right)^k
  \right.
  \\
  &\quad+
  \left.
  \frac{\sqrt{abN}e^{-\frac{Nabr\pi\sqrt{-1}}{2}}}
       {2\sqrt{2r}e^{\frac{\pi\sqrt{-1}}{4}}}
    \sum_{j=1}^{\floor{abr}}
    (-1)^{Nj+j+1}
    \mathcal{R}\bigl(T(a,b);j\bigr)
    \mathcal{CS}\bigl(T(a,b);j\bigr)^{-N/r}
  \right\}.
\end{split}
\end{equation}
If $\pi{r}\sqrt{-1}\in\mathcal{P}$, then
\begin{equation}\label{eq:not_in_P}
\begin{split}
  &J_N\left(T(a,b);\exp\left(\frac{2\pi{r}\sqrt{-1}}{N}\right)\right)
  \\
  =&
  \frac{e^{\left(ab-\frac{a}{b}-\frac{b}{a}\right)\frac{\pi{r}\sqrt{-1}}{2N}}}
       {\sin(\pi{r})}
  \left\{
    \frac{1}{2\sqrt{-1}}
    \sum_{k=0}^{\infty}
    \frac{\gamma_{a,b,r}(k)}{k!}
    \left(\frac{\pi{r}\sqrt{-1}}{2abN}\right)^k
  \right.
  \\
  &\quad
  -
  \frac{\sqrt{N}}
       {\sqrt{2abr}e^{\frac{\pi\sqrt{-1}}{4}}}
  \times
  (-1)^{abr}\sin(ar\pi)\sin(br\pi)
  \\
  &\quad
  \left.
    +
    \frac{\sqrt{abN}e^{-\frac{Nabr\pi\sqrt{-1}}{2}}}
         {2\sqrt{2r}e^{\frac{\pi\sqrt{-1}}{4}}}
    \sum_{j=1}^{abr-1}
    (-1)^{Nj+j+1}
    \mathcal{R}\bigl(T(a,b);j\bigr)
    \mathcal{CS}\bigl(T(a,b);j\bigr)^{-N/r}
  \right\}.
\end{split}
\end{equation}
\end{prop}
\begin{rem}
In \cite{Kashaev/Tirkkonen:ZAPNS2000}, R.~Kashaev and O.~Tirkkonen proved the following asymptotic expansion, which corresponds to the case $r=1$.
See also \cite{Hikami/Kirillov:PHYLB2003}.
\begin{equation*}
\begin{split}
  &J_N\left(T(a,b);\exp\left(\frac{2\pi\sqrt{-1}}{N}\right)\right)
  \\
  =&
  e^{\left(ab-\frac{a}{b}-\frac{b}{a}\right)
     \frac{\pi\sqrt{-1}}{2N}}
  \left\{
    \frac{1}{4}
    \sum_{k=0}^{\infty}\frac{\eta_{a,b}(k+1)}{(k+1)!}
    \left(\frac{\pi\sqrt{-1}}{2abN}\right)^{k}
  \right.
  \\
  &+
  \left.
    \frac{N^{\frac{3}{2}}
          e^{\frac{\pi\sqrt{-1}}{4}}
          e^{-\frac{Nab\pi\sqrt{-1}}{2}}}
         {4\sqrt{2ab}}
    \sum_{j=1}^{ab-1}
    (-1)^{(N-1)j}
    j^2
    \mathcal{R}\bigl(T(a,b);j\bigr)
    \mathcal{CS}\bigl(T(a,b);j\bigr)^{-N}
  \right\},
\end{split}
\end{equation*}
where $\eta_{a,b}(k)$ is the $2k$-th derivative of $z\tau_{T(a,b)}(z)$ at $z=0$.
Note that it coincides with the $(2k-1)$st derivative of $(-1)^{ab+a+b}\tau_{T(a,b)}(z)$ at $z=\pi\sqrt{-1}$.
Note also that Kashaev and Tirkkonen use the $ab$-framed version of the colored Jones polynomial.
In our case we use the $0$-framing version and so we have to renormalize it by framing factor $q^{ab\left(N^2-1\right)/4}$.
(See \cite{Morton:MATPC95} for the framing dependence.)
\par
J.~Dubois and Kashaev \cite{Dubois/Kashaev:MATHA2007} show that $\eta_{a,b}(k)$ is a finite type invariant for every $k$.
They also showed that the square of $\mathcal{R}\bigl(T(a,b);j\bigr)$ is the Reidemeister torsion and $\mathcal{CS}\bigl(T(a,b);j\bigr)$ is the Chern--Simons invariant corresponding to a representation of $\pi_1\left(S^3\setminus{T(a,b)}\right)$ into $SL(2;\C)$.
\end{rem}
\begin{proof}[Proof of Theorem~\ref{thm:Hikami} by using Proposition~\ref{prop:asymptotic}]
If $|r|<1/(ab)$ the second term of \eqref{eq:in_P} vanishes.
So we have
\begin{equation*}
\begin{split}
  J_N\left(T(a,b);\exp\left(\frac{2\pi{r}\sqrt{-1}}{N}\right)\right)
  &=
  \frac{e^{\left(ab-\frac{a}{b}-\frac{b}{a}\right)\frac{\pi{r}\sqrt{-1}}{2N}}}
       {\sin(\pi{r})}
  \frac{1}{2\sqrt{-1}}
  \sum_{k=0}^{\infty}
  \frac{\gamma_{a,b,r}(k)}{k!}
  \left(\frac{\pi{r}\sqrt{-1}}{2abN}\right)^k
  \\
  &\xrightarrow{N\to\infty}
  \frac{\gamma_{a,b,r}(0)}{2\sqrt{-1}\sin(\pi{r})}
  \\
  &=
  \frac{\tau_{T(a,b)}(\pi{r}\sqrt{-1})}{2\sqrt{-1}\sin(\pi{r})}
  \\
  &=
  \frac{\sinh(ar\pi\sqrt{-1})\sinh(br\pi\sqrt{-1})}
       {\sinh(abr\pi\sqrt{-1})\sinh(r\pi\sqrt{-1})}
  \\
  &=
  \frac{1}{\Delta\bigl(T(a,b);\exp(\pi{r}\sqrt{-1})\bigr)}.
\end{split}
\end{equation*}
\par
If $abr=1$, then the third term of \eqref{eq:not_in_P} vanishes and we have
\begin{equation*}
\begin{split}
  &J_N\left(T(a,b);\exp\left(\frac{2\pi{r}\sqrt{-1}}{N}\right)\right)
  \\
  =&
  \frac{e^{\left(ab-\frac{a}{b}-\frac{b}{a}\right)\frac{\pi{r}\sqrt{-1}}{2N}}}
       {\sin(\pi{r})}
  \left\{
    \frac{1}{2\sqrt{-1}}
    \sum_{k=0}^{\infty}
    \frac{\gamma_{a,b,r}(k)}{k!}
    \left(\frac{\pi{r}\sqrt{-1}}{2abN}\right)^k
  \right.
  \\
  &
  \phantom{
    \frac{e^{\left(ab-\frac{a}{b}-\frac{b}{a}\right)\frac{\pi{r}\sqrt{-1}}{2N}}}
         {\sin(\pi{r})}}
  \quad
  -
  \left.
    \frac{\sqrt{N}}
         {\sqrt{2abr}e^{\frac{\pi\sqrt{-1}}{4}}}
    \times
    (-1)^{abr}\sin(ar\pi)\sin(br\pi)
  \right\}
  \\
  \overset{N\to\infty}{\sim}&
  \frac{\sin(ar\pi)\sin(br\pi)}{\sin(\pi{r})}
  \frac{\sqrt{N}}{\sqrt{2}e^{\frac{\pi\sqrt{-1}}{4}}}.
\end{split}
\end{equation*}
\end{proof}
\begin{rem}
If $r$ is real non-integer number with $|r|>1/(ab)$, the sequence $\{J_N\bigl(T(a,b);\exp(2\pi{r}\sqrt{-1}/N)\bigr)\}_{N=2,3,\dots}$ oscillates since the last terms in \eqref{eq:in_P} and \eqref{eq:not_in_P} survive and the other terms converge as above.
Note that in this case $\lim_{N\to\infty}\log\Bigl(J_N\bigl(T(a,b);\exp(2\pi{r}\sqrt{-1}/N)\bigr)\Big)/N=0$ since $J_N\bigl(T(a,b);\exp(2\pi{r}\sqrt{-1}/N)\bigr)$ grows at most polynomially.
\end{rem}
\begin{rem}
Combining the results in \cite{Murakami:INTJM62004} we have the following results about the sequence $\{J_N\bigl(T(a,b);\exp(c/N)\bigr)\}_{N=2,3,\dots}$.
\begin{itemize}
\item
If $\Re{c}>0$ and $|c|>2\pi/(ab)$, or $\Re{c}=0$ and $|c|<2\pi/(ab)$, then it converges and
\begin{equation*}
  \lim_{N\to\infty}
  J_N\bigl(T(a,b);\exp(c/N)\bigr)
  =
  \frac{1}{\Delta\bigl(T(a,b);\exp{c}\bigr)}.
\end{equation*}
\item
If $\Re{c}<0$ and $|c|>2\pi/(ab)$, then it grows exponentially and
\begin{equation*}
  \lim_{N\to\infty}
  \frac{\log{J_N\bigl(T(a,b);\exp(c/N)\bigr)}}{N}
  =
  \left(
    1-\frac{\pi\sqrt{-1}}{abc}-\frac{abc}{4\pi\sqrt{-1}}
  \right)
  \pi\sqrt{-1}.
\end{equation*}
\item
If $c=\pm2\pi\sqrt{-1}/(ab)$, then it grows polynomially.
\item
If $c$ is purely imaginary, $|c|/(2\pi)$ is not an integer, and $|c|>2\pi/(ab)$, then it oscillates.
\end{itemize}
Note that in \cite{Murakami:INTJM62004} the second author stated results only for the case where $\Im{c}>0$, but we can prove the first two formulas above because
\begin{equation*}
  \overline{J_N\bigl(K;\exp(c/N)\bigr)}
  =
  J_N\bigl(K;\exp(\overline{c}/N)\bigr).
\end{equation*}
This was pointed out by A.~Gibson.
\end{rem}
\begin{proof}[Proof of Proposition~\ref{prop:asymptotic}]
From \cite[\S~2]{Murakami:INTJM62004}
(see also \cite{Morton:MATPC95,Kashaev/Tirkkonen:ZAPNS2000}), we have
\begin{equation}\label{eq:int_C}
  J_N\left(T(a,b);e^{\frac{2\pi{r}\sqrt{-1}}{N}}\right)
  =
  \Phi_{a,b,r}(N)\int_{C}\tau_{T(a,b)}(z)e^{Nf_{a,b,r}(z)}\,dz,
\end{equation}
where $C$ is the line that is obtained from the real axis by a $\pi/4$-rotation (Figure~\ref{fig:C}),
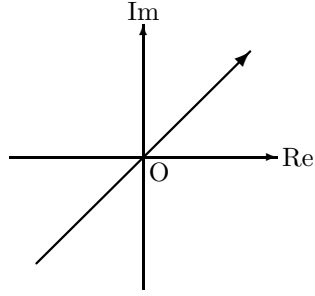
\begin{figure}[h]
\begin{picture}(100,100)(-50,-50)
  \put(-50,  0){\vector(1,0){100}}% x-axis
  \put(  0,-50){\vector(0,1){100}}% y-axis
  \put( 52,  0){\makebox(0,0)[l]{Re}}
  \put(  0, 52){\makebox(0,0)[b]{Im}}
  \thicklines
  \put(-40,-40){\vector(1,1){ 80}}% C
  \put(  2, -2){\makebox(0,0)[tl]{O}}% origin
\end{picture}
\caption{Contour $C$.}
\label{fig:C}
\end{figure}
\begin{equation*}
  f_{a,b,r}(z)
  :=
  ab\left(z-\frac{z^2}{2\pi{r}\sqrt{-1}}\right),
\end{equation*}
and
\begin{equation*}
  \Phi_{a,b,r}(N)
  :=
  g_{a,b,r}\sqrt{N}
  e^{-\left(ab\left(N^2-1\right)+\frac{a}{b}+\frac{b}{a}\right)
     \frac{\pi{r}\sqrt{-1}}{2N}},
\end{equation*}
with
\begin{equation*}
  g_{a,b,r}
  :=
  \frac{\sqrt{ab}}{2\pi\sqrt{2r}e^{\pi\sqrt{-1}/4}\sinh(\pi{r}\sqrt{-1})}.
\end{equation*}
Note that since $r\not\in\Z$, $g_{a,b,r}$ is well-defined and so is the right hand side of \eqref{eq:int_C}.
\par
We first assume that $\pi{r}\sqrt{-1}\not\in\mathcal{P}$.
In this case as in \cite[Page~550]{Murakami:INTJM62004}, we can replace the contour $C$ with $C_r$ where $C_r$ is parallel to $C$ and passes through $\pi{r}\sqrt{-1}$ as \cite[Page~551]{Murakami:INTJM62004}.
\begin{equation}\label{eq:int_C_r}
\begin{split}
  &\int_{C}\tau_{T(a,b)}(z)e^{Nf_{a,b,r}(z)}\,dz
  \\
  =&
  \int_{C_r}
  \tau_{T(a,b)}(z)
  e^{Nf_{a,b,r}(z)}
  \,dz
  +
  2\pi\sqrt{-1}
  \sum_{\stackrel{z_k\in\mathcal{P}}{0<\Im{z_k}<\pi{r}}}
  \Res
  \left(
    \tau_{T(a,b)}(z)e^{Nf_{a,b,r}(z)};z=z_k
  \right)
  \\
  =&
  \int_{C_r}
  \tau_{T(a,b)}(z)
  e^{Nf_{a,b,r}(z)}
  \,dz
  \\
  &+
  2\pi\sqrt{-1}
  \sum_{j=1}^{\floor{abr}}
  (-1)^{j+1}
  \frac{2\sin\left(\frac{j\pi}{a}\right)
         \sin\left(\frac{j\pi}{b}\right)}{ab}
  e^{Nj\pi\sqrt{-1}(1-\frac{j}{2abr})},
\end{split}
\end{equation}
where $\Res(F(x);x=x_0)$ is the residue of $F(x)$ at the point $x_0$.
\begin{rem}
Note that \cite[(2.2)]{Murakami:INTJM62004} is missing $2\pi\sqrt{-1}$ in the second term.
(So \eqref{eq:int_C_r} is correct.)
It does not matter to the results in \cite{Murakami:INTJM62004}, but is essential in the calculation of the current paper.
The second author thanks A.~Gibson, who pointed out this error.
\end{rem}
We will calculate the integral along $C_r$ in \eqref{eq:int_C_r}.
Putting $w:=z-\pi{r}\sqrt{-1}$, we have
\begin{equation*}
\begin{split}
  \int_{C_r}\tau_{T(a,b)}(z)e^{Nf_{a,b,r}(z)}\,dz
  &=
  \int_{C}
  \tau_{T(a,b)}(w+\pi{r}\sqrt{-1})e^{Nf_{a,b,r}(w+\pi{r}\sqrt{-1})}\,dw.
  \\
  &=
  \int_{C}
  \tau_{T(a,b)}(w+\pi{r}\sqrt{-1})
  e^{N\left(-\frac{ab}{2\pi{r}\sqrt{-1}}w^2+\frac{abr\pi\sqrt{-1}}{2}\right)}
  \,dw.
\end{split}
\end{equation*}
But since
\begin{equation*}
\begin{split}
  &\int_{C}
  \tau_{T(a,b)}(w+\pi{r}\sqrt{-1})
  e^{N\left(-\frac{ab}{2\pi{r}\sqrt{-1}}w^2+\frac{abr\pi\sqrt{-1}}{2}\right)}
  \,dw
  \\
  =&
  -
  \int_{-C}
  \tau_{T(a,b)}(-w+\pi{r}\sqrt{-1})
  e^{N\left(-\frac{ab}{2\pi{r}\sqrt{-1}}w^2+\frac{abr\pi\sqrt{-1}}{2}\right)}
  \,dw
  \\
  =&
  \int_{C}
  \tau_{T(a,b)}(-w+\pi{r}\sqrt{-1})
  e^{N\left(-\frac{ab}{2\pi{r}\sqrt{-1}}w^2+\frac{abr\pi\sqrt{-1}}{2}\right)}
  \,dw
\end{split}
\end{equation*}
we have
\begin{multline}\label{eq:int_C_r->C}
  \int_{C_r}\tau_{T(a,b)}(z)e^{Nf_{a,b,r}(z)}\,dz
  \\
  =
  \frac{1}{2}
  \int_{C}
  \left\{
    \tau_{T(a,b)}(w+\pi{r}\sqrt{-1})
    +
    \tau_{T(a,b)}(-w+\pi{r}\sqrt{-1})
  \right\}
  e^{N\left(-\frac{ab}{2\pi{r}\sqrt{-1}}w^2+\frac{abr\pi\sqrt{-1}}{2}\right)}
  \,dw,
\end{multline}
where $-C$ is obtained from $C$ by reversing the orientation.
If we put
\begin{equation*}
  B_{a,b,r}(w)
  :=
  \frac{1}{2}
  \left\{
    \tau_{T(a,b)}(w+\pi{r}\sqrt{-1})
    +
    \tau_{T(a,b)}(-w+\pi{r}\sqrt{-1})
  \right\},
\end{equation*}
we have
\begin{equation}\label{eq:B_expansion}
  B_{a,b,r}(w)
  =
  \sum_{k=0}^{\infty}
  \frac{1}{(2k)!}\frac{d^{2k}\,B_{a,b,r}(0)}{d\,w^{2k}}w^{2k}
  =
  \sum_{k=0}^{\infty}
  \frac{\gamma_{a,b,r}(k)}{(2k)!}w^{2k}
\end{equation}
on $C$ since $B_{a,b,r}$ is an even function and has no poles on $C$.
So we have
\begin{equation}\label{eq:B}
\begin{split}
  \int_{C_r}\tau_{T(a,b)}(z)e^{Nf_{a,b,r}(z)}\,dz
  &=
  \int_{C}B_{a,b,r}(w)
  e^{N\left(-\frac{ab}{2\pi{r}\sqrt{-1}}w^2+\frac{abr\pi\sqrt{-1}}{2}\right)}
  \\
  &=
  \int_{C}
  \left\{
    \sum_{k=0}^{\infty}
    \frac{\gamma_{a,b,r}(k)}{(2k)!}w^{2k}
    e^{N\left(-\frac{ab}{2\pi{r}\sqrt{-1}}w^2+\frac{abr\pi\sqrt{-1}}{2}\right)}
  \right\}
  \,dw
  \\
  &=
  e^{\frac{Nabr\pi\sqrt{-1}}{2}}
  \sum_{k=0}^{\infty}
  \frac{\gamma_{a,b,r}(k)}{(2k)!}
  \int_{C}w^{2k}
  e^{N\left(-\frac{ab}{2\pi{r}\sqrt{-1}}w^2\right)}
  \,dw
  \\
  &=
  2
  e^{\frac{Nabr\pi\sqrt{-1}}{2}}
  \sum_{k=0}^{\infty}
  \frac{\gamma_{a,b,r}(k)}{(2k)!}
  \int_{\llcorner{C}}w^{2k}
  e^{N\left(-\frac{ab}{2\pi{r}\sqrt{-1}}w^2\right)}
  \,dw
\end{split}
\end{equation}
since the integrand is an even function, where $\llcorner{C}$ is the part of $C$ in the first quadrant.
Note that in the third equality we use the fact that the right hand side of \eqref{eq:B_expansion} converges uniformly and so we can exchange the integration and the infinite sum.
Putting $z:=abNw^2/(2\pi{r}\sqrt{-1})$, we have
\begin{equation*}
\begin{split}
  \int_{\llcorner{C}}
  w^{2k}
  e^{N\left(-\frac{ab}{2\pi{r}\sqrt{-1}}w^2\right)}
  \,dw
  &=
  \int_{0}^{\infty}
  \left(\frac{2\pi{r}\sqrt{-1}}{abN}\right)^k
  z^{k}
  e^{-z}
  \sqrt{\frac{2\pi{r}}{abN}}
  e^{\frac{\pi\sqrt{-1}}{4}}
  \times
  \frac{z^{-1/2}}{2}
  \,dz
  \\
  &=
  \sqrt{\frac{\pi{r}}{2abN}}
  \left(\frac{2\pi{r}\sqrt{-1}}{abN}\right)^k
  e^{\frac{\pi\sqrt{-1}}{4}}
  \int_{0}^{\infty}
  z^{k-1/2}
  e^{-z}
  \,dz
  \\
  &=
  \sqrt{\frac{\pi{r}}{2abN}}
  \left(\frac{2\pi{r}\sqrt{-1}}{abN}\right)^k
  e^{\frac{\pi\sqrt{-1}}{4}}
  \Gamma\left(k+\frac{1}{2}\right)
  \\
  &=
  \sqrt{\frac{\pi{r}}{2abN}}
  \left(\frac{2\pi{r}\sqrt{-1}}{abN}\right)^k
  e^{\frac{\pi\sqrt{-1}}{4}}
  \sqrt{\pi}
  \frac{(2k-1)!!}{2^k},
\end{split}
\end{equation*}
where $(2k-1)!!:=1\times3\times5\times\dots\times(2k-1)$.
Therefore \eqref{eq:B} becomes
\begin{equation}\label{eq:C}
\begin{split}
  &
  2\pi
  \sqrt{\frac{r}{2abN}}
  e^{\frac{Nabr\pi\sqrt{-1}}{2}+\frac{\pi\sqrt{-1}}{4}}
  \sum_{k=0}^{\infty}
  \frac{\gamma_{a,b,r}(k)}{(2k)!}
  \left(\frac{2\pi{r}\sqrt{-1}}{abN}\right)^k
  \frac{(2k-1)!!}{2^k}
  \\
  =&
  2\pi
  \sqrt{\frac{r}{2abN}}
  e^{\frac{Nabr\pi\sqrt{-1}}{2}+\frac{\pi\sqrt{-1}}{4}}
  \sum_{k=0}^{\infty}
  \frac{\gamma_{a,b,r}(k)}{k!}
  \left(\frac{\pi{r}\sqrt{-1}}{2abN}\right)^k.
\end{split}
\end{equation}
Finally we have
\begin{equation*}
\begin{split}
  &J_N\left(T(a,b);e^{\frac{2\pi{r}\sqrt{-1}}{N}}\right)
  \\
  =&
  \Phi_{a,b,r}(N)
  \times
  \left\{
    2\pi
    \sqrt{\frac{r}{2abN}}
    e^{\frac{Nabr\pi\sqrt{-1}}{2}+\frac{\pi\sqrt{-1}}{4}}
    \sum_{k=0}^{\infty}
    \frac{\gamma_{a,b,r}(k)}{k!}
    \left(\frac{\pi{r}\sqrt{-1}}{2abN}\right)^k
    \vphantom{
    \frac{2\sin\left(\frac{j\pi}{a}\right)
           \sin\left(\frac{j\pi}{b}\right)}{ab}}
  \right.
  \\
  &+
  \left.
    2\pi\sqrt{-1}
    \sum_{j=1}^{\floor{abr}}
    (-1)^{j+1}
    \frac{2\sin\left(\frac{j\pi}{a}\right)
           \sin\left(\frac{j\pi}{b}\right)}{ab}
    e^{Nj\pi\sqrt{-1}(1-\frac{j}{2abr})}
  \right\}
  \\
  =&
  \frac{\sqrt{abN}}{2\pi\sqrt{2r}e^{\pi\sqrt{-1}/4}\sinh(\pi{r}\sqrt{-1})}
  e^{-\left(ab\left(N^2-1\right)+\frac{a}{b}+\frac{b}{a}\right)
     \frac{\pi{r}\sqrt{-1}}{2N}}
  \\
  &\quad\times
  2\pi
  \sqrt{\frac{r}{2abN}}
  e^{\frac{Nabr\pi\sqrt{-1}}{2}+\frac{\pi\sqrt{-1}}{4}}
  \sum_{k=0}^{\infty}
  \frac{\gamma_{a,b,r}(k)}{k!}
  \left(\frac{\pi{r}\sqrt{-1}}{2abN}\right)^k
  \\
  &+
  \frac{\sqrt{abN}}{2\pi\sqrt{2r}e^{\pi\sqrt{-1}/4}\sinh(\pi{r}\sqrt{-1})}
  e^{-\left(ab\left(N^2-1\right)+\frac{a}{b}+\frac{b}{a}\right)
     \frac{\pi{r}\sqrt{-1}}{2N}}
  \\
  &\quad\times
  2\pi\sqrt{-1}
  \sum_{j=1}^{\floor{abr}}
  (-1)^{j+1}
  \frac{2\sin\left(\frac{j\pi}{a}\right)
         \sin\left(\frac{j\pi}{b}\right)}{ab}
  e^{Nj\pi\sqrt{-1}(1-\frac{j}{2abr})}
  \\
  =&
  \frac{1}{2\sqrt{-1}\sin(\pi{r})}
  e^{\left(ab-\frac{a}{b}-\frac{b}{a}\right)\frac{\pi{r}\sqrt{-1}}{2N}}
  \sum_{k=0}^{\infty}
  \frac{\gamma_{a,b,r}(k)}{k!}
  \left(\frac{\pi{r}\sqrt{-1}}{2abN}\right)^k
  \\
  &+
  \frac{\sqrt{2N}}{\sqrt{abr}e^{\pi\sqrt{-1}/4}\sin(\pi{r})}
  e^{-\left(ab\left(N^2-1\right)+\frac{a}{b}+\frac{b}{a}\right)
     \frac{\pi{r}\sqrt{-1}}{2N}}
  \\
  &\quad\times
  \sum_{j=1}^{\floor{abr}}
  (-1)^{j+1}
  \sin\left(\frac{j\pi}{a}\right)
  \sin\left(\frac{j\pi}{b}\right)
  e^{Nj\pi\sqrt{-1}(1-\frac{j}{2abr})}
  \\
  =&
  \frac{e^{\left(ab-\frac{a}{b}-\frac{b}{a}\right)\frac{\pi{r}\sqrt{-1}}{2N}}}
       {\sin(\pi{r})}
  \left\{
    \frac{1}{2\sqrt{-1}}
    \sum_{k=0}^{\infty}
    \frac{\gamma_{a,b,r}(k)}{k!}
    \left(\frac{\pi{r}\sqrt{-1}}{2abN}\right)^k
  \right.
  \\
  &+
  \frac{\sqrt{2N}}{\sqrt{abr}e^{\pi\sqrt{-1}/4}}
  e^{-Nabr\pi\sqrt{-1}/2}
  \vphantom{\left(\frac{j\pi}{a}\right)}
  \\
  &\quad\times
  \left.
    \sum_{j=1}^{\floor{abr}}
    (-1)^{j+1}
    \sin\left(\frac{j\pi}{a}\right)
    \sin\left(\frac{j\pi}{b}\right)
    e^{Nj\pi\sqrt{-1}(1-\frac{j}{2abr})}
  \right\}.
\end{split}
\end{equation*}
\par
Next assume that $\pi{r}\sqrt{-1}\in\mathcal{P}$.
For simplicity we assume that $r>0$.
In this case we replace the contour $C$ in the integral in \eqref{eq:int_C_r} with $\check{C}_r$, where $\check{C}_r$ is obtained from $C_r$ by adding a small detour below the point $\pi{r}\sqrt{-1}$ as in Figure~\ref{fig:detourdown_C_r}.
\begin{equation}\label{eq:C2}
\begin{split}
  &\int_{C}\tau_{T(a,b)}(z)e^{Nf_{a,b,r}(z)}\,dz
  \\
  =&
  \int_{\check{C}_r}
  \tau_{T(a,b)}(z)
  e^{Nf_{a,b,r}(z)}
  \,dz
  \\
  &+
  2\pi\sqrt{-1}
  \sum_{j=1}^{abr-1}
  (-1)^{j+1}
  \frac{2\sin\left(\frac{j\pi}{a}\right)
         \sin\left(\frac{j\pi}{b}\right)}{ab}
  e^{Nj\pi\sqrt{-1}(1-\frac{j}{2abr})}.
\end{split}
\end{equation}
Note that $abr$ is an integer.
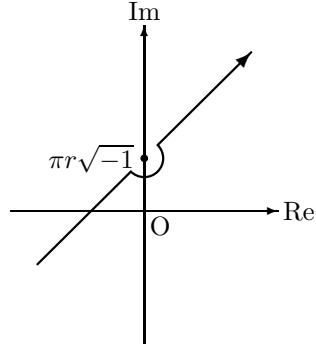
\begin{figure}[h]
\begin{picture}(100,120)(-50,-50)
  \put(-50,  0){\vector(1,0){100}}% x-axis
  \put(  0,-50){\vector(0,1){120}}% y-axis
  \put( 52,  0){\makebox(0,0)[l]{Re}}
  \put(  0, 72){\makebox(0,0)[b]{Im}}
  \thicklines
  \put(-40,-20){\line(1,1){35}}% C:tail
  \put(  5, 25){\vector(1,1){35}}% C:head
  \put(  0, 20){\arc{14.1421356}{-0.785398}{2.35619}}% detour
  \put(  2, -2){\makebox(0,0)[tl]{O}}% origin
  \put(  0, 20){\circle*{2}}% point(\pi{r}\sqrt{-1})
  \put( -4, 20){\makebox(0,0)[r]{$\pi{r}\sqrt{-1}$}}
\end{picture}
\caption{Contour $\check{C}_r$}
\label{fig:detourdown_C_r}
\end{figure}
\par
As in \eqref{eq:int_C_r->C}, we have
\begin{equation}\label{eq:int_check_C_r->check_C}
\begin{split}
  &\int_{\check{C}_r}\tau_{T(a,b)}(z)e^{Nf_{a,b,r}(z)}\,dz
  \\
  =&
  \frac{1}{2}
  \left\{
    \int_{\check{C}}
    \tau_{T(a,b)}(w+\pi{r}\sqrt{-1})
    e^{N\left(-\frac{ab}{2\pi{r}\sqrt{-1}}w^2+\frac{abr\pi\sqrt{-1}}{2}\right)}
    \,dw
  \right.
  \\
  &\quad+
  \left.
    \int_{-\hat{C}}
    \tau_{T(a,b)}(-w+\pi{r}\sqrt{-1})
    e^{N\left(-\frac{ab}{2\pi{r}\sqrt{-1}}w^2+\frac{abr\pi\sqrt{-1}}{2}\right)}
    \,dw
  \right\}
  \\
  =&
  \int_{\check{C}}
  B_{a,b,r}(w)
  e^{N\left(-\frac{ab}{2\pi{r}\sqrt{-1}}w^2+\frac{abr\pi\sqrt{-1}}{2}\right)}
  \,dw
  \\
  &-
  \frac{1}{2}
  \int_{D}
  \tau_{T(a,b)}(-w+\pi{r}\sqrt{-1})
  e^{N\left(-\frac{ab}{2\pi{r}\sqrt{-1}}w^2+\frac{abr\pi\sqrt{-1}}{2}\right)}
  \,dw,
\end{split}
\end{equation}
where $\check{C}$ is the same as $C$ except for a small detour below the origin (Figure~\ref{fig:C_detourdown}), $\hat{C}$ is obtained from $\check{C}$ by a $\pi$-rotation around the origin (Figure~\ref{fig:C_detourup}), $-\hat{C}$ is the same contour as $\hat{C}$ with reversed orientation, and $D$ is a small circle with anticlockwise orientation around the origin.
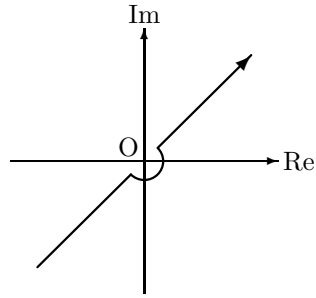
\begin{figure}[h]
\begin{picture}(100,100)(-50,-50)
  \put(-50,  0){\vector(1,0){100}}% x-axis
  \put(  0,-50){\vector(0,1){100}}% y-axis
  \put( 52,  0){\makebox(0,0)[l]{Re}}
  \put(  0, 52){\makebox(0,0)[b]{Im}}
  \thicklines
  \put(-40,-40){\line(1,1){35}}% C:tail
  \put(  5,  5){\vector(1,1){35}}% C:head
  \put(  0,  0){\arc{14.1421356}{-0.785398}{2.35619}}% detour
  \put( -2,  2){\makebox(0,0)[br]{O}}% origin
\end{picture}
\caption{Contour $\check{C}$}
\label{fig:C_detourdown}
\end{figure}
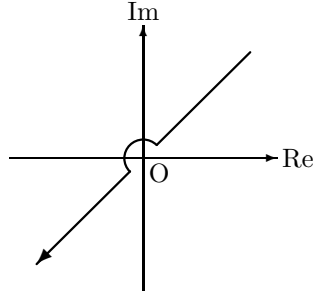
\begin{figure}[h]
\begin{picture}(100,100)(-50,-50)
  \put(-50,  0){\vector(1,0){100}}% x-axis
  \put(  0,-50){\vector(0,1){100}}% y-axis
  \put( 52,  0){\makebox(0,0)[l]{Re}}
  \put(  0, 52){\makebox(0,0)[b]{Im}}
  \thicklines
  \put( -5, -5){\vector(-1,-1){35}}% C:tail
  \put(  5,  5){\line(1,1){35}}% C:head
  \put(  0,  0){\arc{14.1421356}{2.35619}{5.49779}}% detour
  \put(  2, -2){\makebox(0,0)[tl]{O}}% origin
\end{picture}
\caption{Contour $\hat{C}$}
\label{fig:C_detourup}
\end{figure}
Now from the residue theorem, we have
\begin{equation*}
\begin{split}
  &\int_{D}
  \tau_{T(a,b)}(-w+\pi{r}\sqrt{-1})
  e^{N\left(-\frac{ab}{2\pi{r}\sqrt{-1}}w^2+\frac{abr\pi\sqrt{-1}}{2}\right)}
  \,dw
  \\
  =&
  2\pi\sqrt{-1}\times
  e^{\frac{Nabr\pi\sqrt{-1}}{2}}
  \Res\bigl(\tau_{T(a,b)}(-w+\pi{r}\sqrt{-1});w=0\bigr)
  \\
  =&
  2\pi\sqrt{-1}\times
  e^{\frac{Nabr\pi\sqrt{-1}}{2}}
  \times
  \frac{(-1)^{abr}2\sin(ar\pi)\sin(br\pi)}{ab}.
\end{split}
\end{equation*}
\par
Next we calculate the integral along $\check{C}$.
Since $\tau_{a,b,r}\left(w+\pi{r}\sqrt{-1}\right)$ has a simple pole at $0$, we can write
\begin{equation*}
  \tau_{a,b,r}\left(w+\pi{r}\sqrt{-1}\right)
  =
  \frac{\nu_{a,b,r}(-1)}{w}
  +
  \sum_{k=0}^{\infty}\nu_{a,b,c}(k)w^k
\end{equation*}
on $\check{C}$.
Then since
\begin{equation*}
  \tau_{a,b,r}\left(-w+\pi{r}\sqrt{-1}\right)
  =
  -\frac{\nu_{a,b,r}(-1)}{w}
  +
  \sum_{k=0}^{\infty}(-1)^{k}\nu_{a,b,r}(k)w^k
\end{equation*}
we have
\begin{equation*}
  B_{a,b,r}(w)
  =
  \sum_{k=0}^{\infty}\nu_{a,b,r}(2k)w^{2k}
  =
  \sum_{k=0}^{\infty}\frac{\gamma_{a,b,r}}{(2k)!}w^{2k}.
\end{equation*}
\begin{comment}
\begin{equation*}
\begin{split}
  \tau_{a,b,r}(w+\pi{r}\sqrt{-1})
  &=
  \frac{2\sinh(aw+ar\pi\sqrt{-1})\sinh(bw+br\pi\sqrt{-1})}
       {\sinh(abw+abr\pi\sqrt{-1})}
  \\
  &=
  \frac{(-1)^{abr}2\sinh(aw+ar\pi\sqrt{-1})\sinh(bw+br\pi\sqrt{-1})}
       {\sinh(abw)}
\end{split}
\end{equation*}
and so
\begin{multline*}
  B_{a,b,r}(w)
  =
  \frac{(-1)^{abr}}{\sinh(abw)}
  \left\{
    \sinh(aw+ar\pi\sqrt{-1})\sinh(bw+br\pi\sqrt{-1})
  \right.
  \\
  -
  \left.
    \sinh(-aw+ar\pi\sqrt{-1})\sinh(-bw+br\pi\sqrt{-1})
  \right\}.
\end{multline*}
Putting
\begin{equation*}
\begin{cases}
  \tau_1(w):=
  &\sinh(aw+ar\pi\sqrt{-1})\sinh(bw+br\pi\sqrt{-1})
  \\
  &-
  \sinh(-aw+ar\pi\sqrt{-1})\sinh(-bw+br\pi\sqrt{-1}),
  \\
  \tau_2(w):=
  &\sinh(abw).
\end{cases}
\end{equation*}
and using l'H{\^o}pital's rule, we have
\begin{multline*}
  \lim_{w\to0}
  B_{a,b,r}(w)
  =
  \lim_{w\to0}
  \frac{d\,\tau_1(w)/d\,w}{d\,\tau_2(w)/d\,w}
  \\
  =
  \frac{2a\cosh(ar\pi\sqrt{-1})\sinh(br\pi\sqrt{-1})
       +2b\sinh(ar\pi\sqrt{-1})\cosh(br\pi\sqrt{-1})}{ab}.
\end{multline*}
\end{comment}
Therefore in particular the function $B_{a,b}(w)$ is continuous at $w=0$, and so the path $\check{C}$ in \eqref{eq:int_check_C_r->check_C} can be replaced with $C$.
Now \eqref{eq:int_check_C_r->check_C} becomes
\begin{multline*}
  \int_{C}
  B_{a,b,r}(w)
  e^{N\left(-\frac{ab}{2\pi{r}\sqrt{-1}}w^2+\frac{abr\pi\sqrt{-1}}{2}\right)}
  \,dw
  \\
  -
  2\pi\sqrt{-1}\times
  e^{\frac{Nabr\pi\sqrt{-1}}{2}}
  \times
  \frac{(-1)^{abr}\sin(ar\pi)\sin(br\pi)}{ab}.
\end{multline*}
From \eqref{eq:B} and \eqref{eq:C}, this becomes
\begin{multline*}
  2\pi
  \sqrt{\frac{r}{2abN}}
  e^{\frac{Nabr\pi\sqrt{-1}}{2}+\frac{\pi\sqrt{-1}}{4}}
  \sum_{k=0}^{\infty}
  \frac{\gamma_{a,b,r}(k)}{k!}
  \left(\frac{\pi{r}\sqrt{-1}}{2abN}\right)^k
  \\
  -
  2\pi\sqrt{-1}\times
  e^{\frac{Nabr\pi\sqrt{-1}}{2}}
  \times
  \frac{(-1)^{abr}\sin(ar\pi)\sin(br\pi)}{ab}.
\end{multline*}
Therefore \eqref{eq:C2} becomes
\begin{multline*}
  2\pi
  \sqrt{\frac{r}{2abN}}
  e^{\frac{Nabr\pi\sqrt{-1}}{2}+\frac{\pi\sqrt{-1}}{4}}
  \sum_{k=0}^{\infty}
  \frac{\gamma_{a,b,r}(k)}{k!}
  \left(\frac{\pi{r}\sqrt{-1}}{2abN}\right)^k
  \\
  -
  2\pi\sqrt{-1}\times
  e^{\frac{Nabr\pi\sqrt{-1}}{2}}
  \times
  \frac{(-1)^{abr}\sin(ar\pi)\sin(br\pi)}{ab}
  \\
  +
  2\pi\sqrt{-1}
  \sum_{j=1}^{abr-1}
  (-1)^{j+1}
  \frac{2\sin\left(\frac{j\pi}{a}\right)
         \sin\left(\frac{j\pi}{b}\right)}{ab}
  e^{Nj\pi\sqrt{-1}(1-\frac{j}{2abr})}.
\end{multline*}
So we finally have
\begin{equation*}
\begin{split}
  &J_N\left(T(a,b);\exp\left(\frac{2\pi{r}\sqrt{-1}}{N}\right)\right)
  \\
  =&
  \frac{\sqrt{ab}}{2\pi\sqrt{2r}e^{\pi\sqrt{-1}/4}\sinh(\pi{r}\sqrt{-1})}
  \sqrt{N}
  e^{-\left(ab\left(N^2-1\right)+\frac{a}{b}+\frac{b}{a}\right)
     \frac{\pi{r}\sqrt{-1}}{2N}}
  \\
  &
  \times
  \left\{
    2\pi
    \sqrt{\frac{r}{2abN}}
    e^{\frac{Nabr\pi\sqrt{-1}}{2}+\frac{\pi\sqrt{-1}}{4}}
    \sum_{k=0}^{\infty}
    \frac{\gamma_{a,b,r}(k)}{k!}
    \left(\frac{\pi{r}\sqrt{-1}}{2abN}\right)^k
  \right.
  \\
  &\qquad
  -
  2\pi\sqrt{-1}\times
  e^{\frac{Nabr\pi\sqrt{-1}}{2}}
  \times
  \frac{(-1)^{abr}\sin(ar\pi)\sin(br\pi)}{ab}
  \\
  &\qquad
  \left.
    +
    2\pi\sqrt{-1}
    \sum_{j=1}^{abr-1}
    (-1)^{j+1}
    \frac{2\sin\left(\frac{j\pi}{a}\right)
           \sin\left(\frac{j\pi}{b}\right)}{ab}
    e^{Nj\pi\sqrt{-1}(1-\frac{j}{2abr})}
  \right\}
  \\
  =&
  \frac{e^{\left(ab-\frac{a}{b}-\frac{b}{a}\right)\frac{\pi{r}\sqrt{-1}}{2N}}}
       {\sin(\pi{r})}
  \\
  &
  \times
  \left\{
    \frac{1}{2\sqrt{-1}}
    \sum_{k=0}^{\infty}
    \frac{\gamma_{a,b,r}(k)}{k!}
    \left(\frac{\pi{r}\sqrt{-1}}{2abN}\right)^k
  \right.
  \\
  &\qquad
  -
  \frac{\sqrt{N}}
       {\sqrt{2abr}e^{\pi\sqrt{-1}/4}}
  \times
  (-1)^{abr}\sin(ar\pi)\sin(br\pi)
  \\
  &\qquad
  \left.
    +
    \frac{\sqrt{2N}e^{-\frac{Nabr\pi\sqrt{-1}}{2}}}
         {\sqrt{abr}e^{\pi\sqrt{-1}/4}}
    \sum_{j=1}^{abr-1}
    (-1)^{j+1}
    \sin\left(\frac{j\pi}{a}\right)\sin\left(\frac{j\pi}{b}\right)
    e^{Nj\pi\sqrt{-1}(1-\frac{j}{2abr})}
  \right\}.
\end{split}
\end{equation*}
\end{proof}
\section{Connected-sums}
In this section we study connected-sums of copies of the figure-eight knot and torus knots.
%%%%%%%%%%%%%%%%%%%%%%%%%%%%%%%%%%%%%%%%%%%%%%%%%
\begin{comment}
We prepare a couple of lemmas.
\begin{lem}\label{lem:torus_at_xi}
Let $\xi:=\log\bigl((3+\sqrt{5})/2\bigr)$ as before.
Then the sequence $\{J_N\bigl(T(a,b);\exp(\xi/N)\bigr)\}_{N=2,3,\dots}$ converges to $1/\Delta\bigl(T(a,b);\exp{\xi}\bigr)$
\end{lem}
\begin{proof}
From Theorem~\ref{thm:Hikami}, if $|r|<1/(ab)$ then the sequence $\{J_N\bigl(T(a,b);\exp(2\pi{r}\sqrt{-1}/N)\bigr)\}_{N=2,3,\dots}$ converges to the inverse of the Alexander polynomial evaluated at $\exp(2\pi{r}\sqrt{-1})$.
In this case $|r|=\xi/(2\pi)=0.15317\dots<1/6=0.16666\dots$.
The proof is complete because $ab\ge6$ for integers $a$ and $b$ with $a>1$, $b>1$ and $(a,b)=1$.
\end{proof}
\begin{lem}\label{lem:fig8_at_torus}
If $|y|<\pi/2$ for a real number $y$, then the sequence $\{J_N\bigl(E;\exp(y\sqrt{-1}/N)\bigr)\}_{N=2,3,\dots}$ converges to $1/\Delta\bigl(E;\exp(y\sqrt{-1})\bigr)$, where $E$ is the figure-eight knot.
In particular $\{J_N\bigl(E;\exp(2\pi\sqrt{-1}/(abN))\bigr)\}_{N=2,3,\dots}$ converges to $1/\Delta\bigl(E;\exp(2\pi\sqrt{-1}/(ab))\bigr)$ if $|ab|\ge6$.
\end{lem}
\begin{proof}
This follows from the proof of the main result in \cite{Murakami:JPJGT2007}, which was suggested by A.~Gibson.
\end{proof}
\end{comment}
%%%%%%%%%%%%%%%%%%%%%%%%%%%%%%
We have the following proposition that supports Conjecture~\ref{conj:polynomial}.
\begin{prop}
For connected-sums of copies of the figure-eight knot and torus knots, Conjecture~\ref{conj:polynomial} holds.
\end{prop}
\begin{proof}
First note that $\xi=0.96242\cdots$ and so we have
\begin{equation*}
  \frac{2\pi}{2\times3}>\xi>\frac{2\pi}{ab}
\end{equation*}
for any pair $(a,b)$ of coprime integers with $ab>6$.
Put
\begin{equation*}
  K
  :=
  T(2,3)^{k_1}\sharp
  {\overline{T(2,3)}}^{k_2}\sharp
  E^{l}\sharp
  \left(\mathop{\sharp}_{i=1}^{m_1}T\left(a_i,b_i\right)\right)\sharp
  \left(\mathop{\sharp}_{j=1}^{m_2}\overline{T\left(a_j,b_j\right)}\right),
\end{equation*}
where $a_ib_i>6$, $a_jb_j>6$, $L^{k}$ denotes the connected-sum of $k$ copies of a knot $L$, and $\overline{T\left(a_j,b_j\right)}$ is the mirror image of $T\left(a_j,b_j\right)$.
We have
\begin{equation*}
  \min\{|z|\mid z\in\Lambda(K)\}
  =
  \begin{cases}
    \frac{\pi}{3}
    &\quad\text{if $l=m_1=m_2=0$},
    \\
    \xi
    &\quad\text{if $m_1=m_2=0$},
    \\
    \min\left\{\frac{2\pi}{a_ib_i},\frac{2\pi}{a_jb_j}\right\}
        _{i=1,2,\dots,m_1,j=1,2,\dots,m_2}
    &\quad\text{otherwise}.
  \end{cases}
\end{equation*}
\par
Therefore if $l=m_1=m_2=0$, the proposition is clear.
If $m_1=m_2=0$, then we have
\begin{equation*}
\begin{split}
  &J_N\left(T(2,3)^{k_1}\sharp T(2,3)^{k_2}\sharp E^l;\exp(\xi/N)\right)
  \\
  =&
  J_N\bigl(T(2,3);\exp(\xi/N)\bigr)^{k_1}
  \times
  J_N\bigl(\overline{T(2,3)};\exp(\xi/N)\bigr)^{k_2}
  \times
  J_N\bigl(E;\exp(\xi/N)\bigr)^{l}.
\end{split}
\end{equation*}
Since we can prove that the sequence $\left\{J_N\bigl(T(2,3);\exp(\xi/N)\bigr)\right\}_{N=2,3,\dots}$ converges by using a similar technique (see Proposition~\ref{prop:purely_imaginary} below), the conjecture is true in this case.
\par
From \cite[Theorem~1.1]{Murakami:JPJGT2007}, we know that $\left\{J_N\bigl(T(2,3);\exp(c/N)\bigr)\right\}_{N=2,3,\dots}$ converges if $c$ is purely imaginary and $|c|<\pi/3$.
Since $2\pi/(a_ib_i)<\pi/3$ and $2\pi/(a_jb_j)<\pi/3$ for $i=1,2,\dots,m_1$ and $j=1,2,\dots,m_2$, we can also prove the other case.
\end{proof}
%%%%%%%%%%%%%%%%%%%%%%%%%%%%%%%%%%%%%%%%%%%%%%%%%%%%%%%
%\setcounter{section}{1}
%\setcounter{thm}{0}
%\renewcommand{\thesection}{\Alph{section}}
\appendix
\section{}
In this appendix we study the sequence $\{J_N\bigl(T(a,b);\exp(c/N)\bigr)\}_{N=2,3,\dots}$ for a real number $c$ with $|c|<2\pi/(ab)$.
First of all, note that if $c=0$ then $J_N\bigl(T(a,b);\exp(c/N)\bigr)=1$ for any $N$.
\par
If $-2\pi/(ab)<c<0$, we define the contour $C$ as a $\varphi$-rotation of the real axis with $\varphi=\pi/4+\delta$ for a small $\delta>0$ so that \eqref{eq:int_C} holds with $r:=c/(2\pi\sqrt{-1})$.
Here we choose $\delta$ small enough so that $|c|\tan{\varphi}<2\pi/(ab)$.
Let $C_r$ be the parallel translation of $C$ that passes through $c/2=\pi{r}\sqrt{-1}$.
Since $C_r$ crosses the imaginary axis at $\pi{r}\tan{\varphi}$, there is no pole of $\tau_{T(a,b)}$ between $C$ and $C_r$.
So from the same argument of \eqref{eq:int_C_r} we have
\begin{equation*}
  \int_{C}\tau_{T(a,b)}(z)e^{N\,f_{a,b,r}(z)}\,dz
  \\
  =
  \int_{C_r}
  \tau_{T(a,b)}(z)
  e^{Nf_{a,b,r}(z)}
  \,dz.
\end{equation*}
Therefore from the calculation in \cite{Murakami:INTJM62004}, we have
\begin{equation*}
  \lim_{N\to\infty}
  J_N\bigl(T(a,b);\exp(c/N)\bigr)
  =
  \frac{1}{\Delta\bigl(T(a,b);\exp(c)\bigr)}.
\end{equation*}
\par
If $0<c<2\pi/(ab)$, then we can choose the real axis as the contour $C$.
Since it is clear that there is no pole of $\tau_{T(a,b)}$ between $C$ and $C_r$, we have the same formula.
\par
Thus we have proved
\begin{prop}\label{prop:purely_imaginary}
For a real number $c$ with $|c|<2\pi/(ab)$ the sequence $\{J_N\bigl(T(a,b);\exp(c/N)\bigr)\}_{N=2,3,\dots}$ converges to $1/\Delta\bigl(T(a,b);\exp{c}\bigr)$.
\end{prop}
\par
\begin{rem}
By similar arguments we can determine the asymptotic behavior of the sequence $\{J_N\bigl(T(a,b);\exp(c/N)\bigr)\}_{N=2,3,\dots}$ for any $c\in\C$ as indicated in Figure~\ref{fig:convergence}.
See our forthcoming paper for more details.
\begin{figure}[h]
\includegraphics{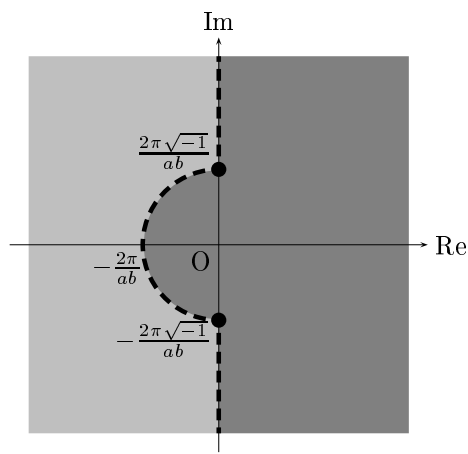}
\caption{This picture shows the $c$-plane, where the dark shadow indicates the region where $J_N\bigl(T(a,b);\exp(c/N)\bigr)$ converges, the light shadow indicates the region where it grows exponentially, the two dots indicate the points where it grows polynomially, and the broken line indicates the region where it oscillates except for integral multiples of $2\pi\sqrt{-1}$.}
\label{fig:convergence}
\end{figure}
\end{rem}
%%%%%%%%%%%%%%%%%%%%%%%%%%%%%%%%%%%%%%%%%%%%%%%%%%%%%%%%%%%%%%%%%%%%%%%%%%%%%
\bibliography{mrabbrev,hitoshi}
\bibliographystyle{amsplain}
\end{document}